\documentclass[a4paper,oneside,11pt]{article}

\usepackage{enumitem}
\usepackage{amsmath}
\usepackage{amssymb,amsfonts}
\usepackage{amsthm}
\usepackage[colorlinks,
	citecolor=darkblue,
	linkcolor=darkblue,
	urlcolor=darkblue]{hyperref}
\usepackage[ruled,vlined,linesnumbered,algosection,algo2e]{algorithm2e} 
\usepackage[capitalize,noabbrev,nameinlink]{cleveref} 

\usepackage{xcolor}
\definecolor{darkblue}{RGB}{0,102,153}

\theoremstyle{plain}
\newtheorem{thm}{Theorem}[section]

\newtheorem{lem}[thm]{Lemma}
\newtheorem{prop}[thm]{Proposition}

\theoremstyle{definition}
\newtheorem{defn}[thm]{Definition}
\theoremstyle{remark}

\newcommand{\coloneqq}{:=}
\DeclareMathOperator*{\minimize}{minimize}
\DeclareMathOperator{\stt}{subject~to}
\newcommand{\func}[3]{#1 \colon #2 \to #3}
\newcommand{\norm}[1]{\lVert #1 \rVert}
\newcommand{\innprod}[2]{\langle #1, #2 \rangle}
\newcommand{\N}{\mathbb{N}}
\newcommand{\R}{\mathbb{R}}
\newcommand{\Rinf}{\overline{\R}}
\newcommand{\XX}{\R^n}
\newcommand{\YY}{\R^m}
\DeclareMathOperator{\dom}{dom}

\newcommand{\email}[1]{\textsc{email} \href{mailto:#1}{#1}}
\newcommand{\orcid}[1]{\textsc{orcid} \href{https://orcid.org/#1}{#1}}
\newcommand{\amsmsc}[1]{\href{http://www.ams.org/mathscinet/msc/msc2020.html?t=#1}{#1}}

\newcommand{\KeywordsAnd}{\and}
\newcommand{\keywords}[1]{\par\noindent{\def\and{\unskip,\ }{\bf\small Keywords. }\small#1}\par}
\newcommand{\subclass}[1]{\par\noindent{\def\and{\unskip,\ }{\bf\small AMS MSC. }\small#1}\par}

\crefname{line}{Step}{Steps}
\SetKw{KwOr}{or}

\newcommand{\regip}{{}\textsc{RegIP}{}}
\newcommand{\percival}{{}\textsc{Percival}{}}
\newcommand{\ipopt}{{}\textsc{Ipopt}{}}
\newcommand{\ipoptOption}[1]{\texttt{#1}}

\newcommand{\TheAuthor}{Alberto De~Marchi}
\newcommand{\TheEmail}{alberto.demarchi@unibw.de}
\newcommand{\TheOrcid}{0000-0002-3545-6898}
\newcommand{\TheAffiliation}{%
	Universit{\"a}t der Bundeswehr M{\"u}nchen,
	Department of Aerospace Engineering,
	Institute of Applied Mathematics and Scientific Computing,
	Neubiberg/Munich, Germany%
}
\newcommand{\TheAcknowledgements}{%
	I gratefully acknowledge the support of Ryan Loxton and the Centre for Optimisation and Decision Science for giving me the opportunity to visit Curtin University.
	I would also like to thank Hoa T. Bui for her friendly hospitality and lively discussions during this time in Perth.
}
\newcommand{\ZenodoCodeDoi}{10.5281/zenodo.7109904}
\newcommand{\TheTitle}{Regularized Interior Point Methods for Constrained Optimization and Control}
\newcommand{\TheKeywords}{%
	Nonlinear programming\KeywordsAnd
	optimization-based control\KeywordsAnd
	interior point methods\KeywordsAnd
	augmented Lagrangian%
}
\newcommand{\TheAMSsubj}{%
	\amsmsc{65K05}\and
	\amsmsc{90C06}\and
	\amsmsc{90C26}\and
	\amsmsc{90C30}
}

\begin{document}
	
\title{\bfseries \TheTitle}
\author{\TheAuthor\thanks{\TheAffiliation. \email{\TheEmail}, \orcid{\TheOrcid}.}}
\date{}

\maketitle

\begin{abstract}
	Regularization and interior point approaches offer valuable perspectives to address constrained nonlinear optimization problems in view of control applications.
This paper discusses the interactions between these techniques and proposes an algorithm that synergistically combines them.
Building a sequence of closely related subproblems and approximately solving each of them, this approach inherently exploits warm-starting, early termination, and the possibility to adopt subsolvers tailored to specific problem structures.
Moreover, by relaxing the equality constraints with a proximal penalty, the regularized subproblems are feasible and satisfy a strong constraint qualification by construction, allowing the safe use of efficient solvers.
We show how regularization benefits the underlying linear algebra and a detailed convergence analysis indicates that limit points tend to minimize constraint violation and satisfy suitable optimality conditions.
Finally, we report on numerical results in terms of robustness, indicating that the combined approach compares favorably against both interior point and augmented Lagrangian codes.
\end{abstract}
\subclass{\TheAMSsubj}

\section{Introduction}
Mathematical optimization plays an important role in model-based and data-driven control systems, forming the basis for advanced techniques such as optimal control, nonlinear model predictive control (MPC) and parameters estimation.
Significant research effort on computationally efficient real-time optimization algorithms contributed to the success of MPC over the years and yet the demand for fast and reliable methods for a broad spectrum of applications is growing; see \cite{sopasakis2020open,saraf2022efficient} and references therein.
In order to tackle these challenges, it is desirable to have an algorithm that benefits from warm-starting information, can cope with infeasibility, is robust to problem scaling, and exploits the structure of the problem.
In order to reduce computations and increase robustness, a common approach is to relax the requirements on the solutions, in terms of optimality, constraint violation, or both \cite{diehl2009efficient,saraf2022efficient}.
In this work, we propose to address such features by combining proximal regularization and interior point techniques, for developing a stabilized, efficient and robust numerical method.
We advocate for this strategy by bringing together and combining a variety of ideas from the nonlinear programming literature.

Let us consider the constrained nonconvex problem
\begin{align}\label{eq:P}\tag{P}
	\minimize_x\quad&f(x) \\
	\stt\quad&c(x) = 0 ,\qquad
	x \geq 0 ,\nonumber
\end{align}
where functions $\func{f}{\XX}{\R}$ and $\func{c}{\XX}{\YY}$ are (at least) continuously differentiable.
Problems with general inequality constraints on $x$ and $c(x)$ can be reformulated as \eqref{eq:P} by introducing auxiliary variables.
Nonlinear programming (NLP) problems such as \eqref{eq:P} have been extensively studied and there exist several approaches for their numerical solution.
Interior point (IP) \cite{vanderbei1999interior,waechter2006implementation},
penalty and augmented Lagrangian (AL) \cite{conn1991globally,andreani2008augmented,birgin2014practical}
and sequential programming \cite{fiacco1968nonlinear}
schemes are predominant ideas and have been re-combined in many ways \cite{curtis2012penalty,birgin2016sequential,armand2019rapid}.

Starting from linear programming, IP methods had a significant impact on the field of mathematical optimization \cite{gondzio2012interior}.
By solving a sequence of barrier subproblems, they can efficiently handle inequality constraints and scale well with the problem size.
The state-of-the-art solver \ipopt{}, described by \cite{waechter2006implementation}, is an emblem of this remarkable success.
However, relying on Newton-type schemes for approximately solving the subproblems, IP algorithms may suffer degeneracy and lack of constraint qualifications if suitable counter-mechanisms are not implemented.
On the contrary, proximal techniques naturally cope with these scenarios thanks to their inherent regularizing action.
Widely investigated in the convex setting \cite{rockafellar1976monotone}, their favorable properties have been exploited to design (primal-dual) stabilized methods building on the proximal point algorithm \cite{friedlander2012primal,liaomcpherson2020fbstab,demarchi2022qpdo}.
The analysis of their close connection with the AL framework \cite{rockafellar1974augmented} led to the development of techniques applicable to more general problems \cite{ma2018stabilized,demarchi2021augmented,potschka2021sequential}.

The combination of IP and proximal strategies has been successfully leveraged in the context of convex quadratic programming \cite{altman1999regularized,cipolla2022proximal} and for linear \cite{dehghani2020regularized} and nonlinear \cite{siqueira2018regularized} least-squares problems.
With this work we address general NLPs and devise a method for their numerical solution, which can be seen as an extension of a regularized Lagrange--Newton method to handle bound constraints via a barrier function \cite{demarchi2021augmented}, or as a proximally stabilized IP algorithm, generalizing the ideas put forward by \cite{cipolla2022proximal}.%
\footnote{Although beyond the scope of this paper, topics such as infeasibility detection \cite{armand2019rapid}, ill-posed subproblems \cite{andreani2008augmented}, and feasibility restoration \cite[\S 3.3]{waechter2006implementation} are of great relevance in this context.
Appropriate mechanisms should be incorporated in practical implementations, as they can significantly improve the performances.}

\paragraph*{Outline}
The paper is organized as follows.
In \cref{sec:stationarity_concepts} we provide and comment on some relevant optimality notions.
The methodology is discussed in \cref{sec:methods} detailing the proposed algorithm, whose convergence properties are investigated in \cref{sec:convergence}.
We report numerical results on benchmark problems in \cref{sec:numerics} and conclude the paper in \cref{sec:conclusion}.

\paragraph*{Notation}
With $\N$, $\R$, and $\Rinf \coloneqq \R \cup \{\infty\}$ we denote the natural, real and extended real numbers, respectively.
We denote the set of real vectors of dimension $n\in\N$ as $\R^n$; a real matrix with $m\in\N$ rows and $n\in\N$ columns as $A \in \R^{m\times n}$ and its transpose as $A^\top \in \R^{n\times m}$.
For a vector $a \in \R^n$, its $i$-th element is $a_i$ and its squared Euclidean norm is $\|a\|^2 = a^\top a$.
A vector or matrix with all zero elements is represented by $0$.
The gradient of a function $\func{f}{\XX}{\R}$ at a point $\bar{x} \in \R^n$ is denoted by $\nabla f(\bar{x}) \in \XX$; the Jacobian of a vector function $\func{c}{\XX}{\YY}$ by $\nabla c(\bar{x}) \in \R^{m\times n}$.

\section{Optimality and Stationarity}\label{sec:stationarity_concepts}
In this section we introduce some optimality concepts, following \cite[Ch. 3]{birgin2014practical} and \cite[\S 2]{demarchi2022interior}.

\begin{defn}[Feasibility]
	Relative to \eqref{eq:P}, we say a point $x^\ast\in\XX$ is \emph{feasible} if $x^\ast \geq 0$ and $c(x^\ast) = 0$; it is \emph{strictly feasible} if additionally $x^\ast > 0$.
\end{defn}
\begin{defn}[Approximate KKT stationarity]\label{defn:epsKKTstationary}
	Relative to \eqref{eq:P}, a point $x^\ast\in\XX$ is \emph{$\varepsilon$-KKT stationary} for some $\varepsilon \geq 0$ if there exist multipliers $y^\ast\in\YY$ and $z^\ast\in\XX$ such that
	\begin{subequations}\label{eq:epsKKTstationary}
		\begin{align}
			\norm{ \nabla f(x^\ast) + \nabla c(x^\ast)^\top y^\ast + z^\ast } {}\leq{}& \varepsilon , \label{eq:epsKKTstationary:x}\\
			\norm{ c(x^\ast) } {}\leq{}& \varepsilon , \\
			x^\ast \geq - \varepsilon ,\quad
			z^\ast \leq \varepsilon ,\quad
			\min\{x^\ast,-z^\ast\} {}\leq{}& \varepsilon . \label{eq:epsKKTstationary:z}
		\end{align}
	\end{subequations}
	When $\varepsilon = 0$, the point $x^\ast$ is said \emph{KKT stationary}.
\end{defn}
Notice that \eqref{eq:epsKKTstationary:z} provides a condition modeling the approximate satisfaction of the (elementwise) complementarity condition $\min\{x,-z\} = 0$ within some tolerance $\varepsilon \geq 0$.
The IP algorithm discussed in \cref{sec:methods} satisfies a stronger version of these conditions, since the iterates it generates meet the constraints $x \geq 0$ and $z \leq 0$ by construction.
Furthermore, we point out that the condition $\min\{x_i,-z_i\}\leq \varepsilon$ is analogous to $- x_i z_i \leq \varepsilon$, more typical for interior point methods, but does not depend on a specific barrier function, e.g., the logarithmic barrier in \cite[\S 2.1]{waechter2006implementation}.

We shall consider the limiting behavior of approximate KKT stationary points when the tolerance $\varepsilon$ vanishes.
In fact, having $x^k \to x^\ast$ with $x^k$ $\varepsilon_k$-KKT stationary for \eqref{eq:P} and $\varepsilon_k \searrow 0$ does not guarantee KKT stationarity of a limit point $x^\ast$ of $\{x^k\}$.
This issue raises the need for defining KKT stationarity in an asymptotic sense \cite[Def. 3.1]{birgin2014practical}.
\begin{defn}[Asymptotic KKT stationarity]\label{defn:asympKKTstationary}
	Relative to \eqref{eq:P}, a feasible point $x^\ast\in\XX$ is \emph{AKKT stationary} if there exist sequences $\{x^k\}, \{z^k\}\subset\XX$, and $\{y^k\}\subset\YY$ such that $x^k \to x^\ast$ and
	\begin{subequations}\label{eq:asympKKTstationary}
		\begin{align}
			\nabla f(x^k) + \nabla c(x^k)^\top y^k + z^k {}\to{}& 0 , \label{eq:asympKKTstationary:x} \\
			\min\{ x^k, -z^k \} {}\to{}& 0 . \label{eq:asympKKTstationary:z}
		\end{align}
	\end{subequations}
\end{defn}
Any local minimizer $x^\ast$ for \eqref{eq:P} is AKKT stationary, independently of constraint qualifications \cite[Thm 3.1]{birgin2014practical}.

\section{Approach and Algorithm}\label{sec:methods}
The methodology presented in this section builds upon the AL framework, interpreted as a proximal point scheme in the nonconvex regime, and IP methods.
The basic idea is to construct a sequence of proximally regularized subproblems and to approximately solve each of them as a single barrier subproblem, effectively merging the AL and IP outer loops.
Reduced computational cost can be achieved with an effective warm-starting of the IP iterations and with the tight entanglement of barrier and proximal penalty strategies, by monitoring and updating the parameters' values alongside with the inner tolerance.

A classical approach is to consider a sequence of bound-constrained Lagrangian (BCL) subproblems \cite{conn1991globally,birgin2014practical}
\begin{equation}\label{eq:Pbcl}
	\minimize_{x\geq 0}\quad f(x) + \frac{1}{2 \rho_k} \| c(x) + \rho_k \hat{y}^k \|^2
\end{equation}
where $\rho_k > 0$ and $\hat{y}^k \in \YY$ are some given penalty parameter and dual estimate, respectively.
The nonlinearly-constrained Lagrangian (NCL) scheme \cite{ma2018stabilized} considers equality-constrained subproblems by introducing an auxiliary variable $s \in \YY$ and the constraint $c(x) = s$.
Analogously, a proximal point perspective yields the equivalent reformulation
\begin{align}\label{eq:Pproxy}
	\minimize_{x ,\, \lambda}\quad&f(x) + \frac{\rho_k}{2} \| \lambda \|^2 \\
	\stt\quad&c(x) + \rho_k (\hat{y}^k - \lambda) = 0 ,\qquad
	x \geq 0 , \nonumber
\end{align}
recovering the dual regularization term obtained, e.g., by \cite{potschka2021sequential,demarchi2021augmented,demarchi2022qpdo}.
By construction, these regularized subproblems are always feasible and satisfy a strong constraint qualification, namely the LICQ, at all points.

The regularized subproblems \eqref{eq:Pbcl}--\eqref{eq:Pproxy} can be numerically solved via IP algorithms.
Let us consider a barrier parameter $\mu_k > 0$ and barrier functions $\func{b_i}{\R}{\Rinf}$, $i=1,\ldots,n$, each with domain $\dom b_i = (0,\infty)$, and such that $b_i(t) \to \infty$ as $t\to 0^+$ and $b_i^\prime \leq 0$.
Exemplarily, the logarithmic function $x \mapsto - \ln(x)$ is one of such barrier functions.
Other choices can be considered as well, e.g., to handle bilateral constraints \cite{bertolazzi2007real}.
We collect these barrier functions to define $\func{b}{\XX}{\Rinf}$, $b \colon x \mapsto \sum_{i=1}^n b_i(x_i)$, whose domain is $\dom b = (0,\infty)^n$.
Thus, analogously to \cite{armand2019rapid}, a barrier counterpart for the BCL subproblem \eqref{eq:Pbcl} reads
\begin{equation}\label{eq:Pbarrier_BCL}
	\minimize_x \quad f(x) + \frac{1}{2 \rho_k} \| c(x) + \rho_k \hat{y}^k \|^2 + \mu_k b(x) ,
\end{equation}
whereas for the constrained subproblem \eqref{eq:Pproxy} this leads to
\begin{align}\label{eq:PproxyBarrier}
	\minimize_{x ,\, \lambda}\quad&f(x) + \frac{\rho_k}{2} \| \lambda \|^2 + \mu_k b(x) \\
	\stt\quad&c(x) + \rho_k (\hat{y}^k - \lambda) = 0 , \nonumber
\end{align}
which is a regularized version of \cite[Eq. 3]{waechter2006implementation} and reminiscent of \cite[Eq. 2]{birgin2016sequential}.
It should be stressed that, in stark contrast with classical AL and IP schemes, we intend to find an (approximate) solution to the regularized subproblem \eqref{eq:Pproxy} by (approximately) solving only one barrier subproblem \eqref{eq:PproxyBarrier}.
Inspired by \cite{curtis2012penalty,cipolla2022proximal}, our rationale is to drive $\rho_k, \mu_k$ and the inner tolerance $\epsilon_k$ concurrently toward zero, effectively knitting together proximal and barrier strategies.

It should be noted that a primal (Tikhonov-like) regularization term is not explicitly included in \eqref{eq:Pbcl}--\eqref{eq:PproxyBarrier}.
In fact, the original objective $f$ could be replaced by a (proximal) model of the form $x \mapsto f(x) + \frac{\sigma_k}{2} \| x - \hat{x}^k \|^2$, with some given primal regularization parameter $\sigma_k \geq 0$ and reference point $\hat{x}^k \in \XX$.
However, as this term can be interpreted as an inertia correction, we prefer the subsolver to account for its contribution; cf. \cite[\S 3.1]{waechter2006implementation}.
In this way, the subsolver can surgically tune the primal correction term as needed, possibly improving the convergence speed, and surpassing the issue that suitable values for $\sigma_k$ are unknown a priori.

\begin{algorithm2e}
	\DontPrintSemicolon
	\caption{Regularized interior point method for general nonlinear programs \eqref{eq:P}}%
	\label{alg:RIPM}%
	\KwData{$\epsilon_0, \rho_0, \mu_0 > 0$, $\kappa_\rho, \kappa_\mu, \kappa_\epsilon \in (0,1)$, $\theta_\rho, \theta_\mu \in [0,1)$, $Y \subset \YY$ nonempty bounded, $\varepsilon > 0$}
	\KwResult{$\varepsilon$-KKT stationary point $x^\ast$ with $y^\ast$, $z^\ast$}
	
	\For{$k = 0,1,2,\ldots$}{
		\label{step:yhat}Select $\hat{y}^k \in Y$\;
		\label{step:xy}Find an $\epsilon_k$-KKT stationary point $(x^k,\lambda^k)$ for \eqref{eq:PproxyBarrier}, with multiplier $y^k$\;
		\label{step:z}Set $z^k \gets \mu_k \nabla b(x^k)$\;
		\If{$(x^k, y^k, z^k)$ $\mathrm{satisfies}$ \eqref{eq:epsKKTstationary}}{
			\Return $(x^\ast,y^\ast,z^\ast) \gets (x^k,y^k,z^k)$\;
		}
		Set $C^k \gets \norm{ c(x^k) }$ and $V^k \gets \norm{ \min\{x^k,-z^k\} }$\;
		\If{$k=0$ \KwOr $C^k \leq \max\{ \varepsilon, \theta_\rho C^{k-1}\}$ \label{step:updatePenalty:if}}{
			set $\rho_{k+1} \gets \rho_k$, \textbf{else} select $\rho_{k+1} \in (0,\kappa_\rho \rho_k]$\label{step:updatePenalty:failed}\;
		}
		\If{$k=0$ \KwOr $V^k \leq \max\{ \varepsilon, \theta_\mu V^{k-1}\}$ \label{step:updateBarrier:if}}{
			set $\mu_{k+1} \gets \mu_k$, \textbf{else} select $\mu_{k+1} \in (0,\kappa_\mu \mu_k]$\label{step:updateBarrier:failed}\;
		}
		Set $\epsilon_{k+1} \gets \max\{ \varepsilon, \kappa_\epsilon \epsilon_k \}$\;
	}
\end{algorithm2e}

The overall procedure is detailed in \cref{alg:RIPM}.
At every outer iteration, indexed by $k$, \cref{step:xy} requires to compute an approximate stationary point, with the associated Lagrange multiplier, for the regularized barrier subproblem \eqref{eq:PproxyBarrier}.
As the dual estimate $\hat{y}^k$ is selected from some bounded set $Y\subset\YY$ at \cref{step:yhat}, the AL scheme is \emph{safeguarded} and has stronger global convergence properties \cite[Ch. 4]{birgin2014practical}.
The assignment of $z^k$ at \cref{step:z} follows from comparing and matching the stationarity conditions for \eqref{eq:P} and \eqref{eq:PproxyBarrier}.
After checking termination, we monitor progress in constraint violation and complementarity, based on \eqref{eq:epsKKTstationary}, and update parameters $\rho_k$ and $\mu_k$ accordingly, as well as the inner tolerance $\varepsilon_k$.
At \cref{step:updatePenalty:if,step:updateBarrier:if} we consider relaxed conditions for \emph{satisfactory} feasibility and complementarity as it is preferable to have the sequences $\{\rho_k\}$, $\{\mu_k\}$, and $\{\epsilon_k\}$ bounded away from zero, in order to avoid unnecessary ill-conditioning and tight tolerances.
Sufficient conditions to guarantee boundedness of the penalty parameter $\{\rho_k\}$ away from zero are given, e.g., by \cite[\S 5]{andreani2008augmented}.
Remarkably, as established by \cref{lem:asympComplementarity} in \cref{sec:convergence}, there is no need for the barrier parameter $\mu_k$ to vanish in order to achieve $\varepsilon$-complementarity in the sense of \eqref{eq:epsKKTstationary:z}, for $\varepsilon > 0$.

We shall mention that considering equivalent yet different subproblem formulations may affect the practical performance of the subsolver.
It is enlightening to pinpoint the effect of the dual regularization in \eqref{eq:PproxyBarrier} and to appreciate its interactions with the linear algebra routines used to solve the linear systems arising in Newton-type methods.
Although \eqref{eq:PproxyBarrier} has more (possibly many more) variables than \eqref{eq:Pbarrier_BCL}, a simple reordering yields matrices with the same structure \cite{demarchi2021augmented,potschka2021sequential}.
Let us have a closer look.
Defining the Lagrangian function $\mathcal{L}_k(x,y) \coloneqq f(x) + \mu_k b(x) + \innprod{y}{c(x)}$, the stationarity condition for \eqref{eq:Pbarrier_BCL} reads $0 = \nabla_x \mathcal{L}_k( x, y_k(x) )$, where $y_k(x) \coloneqq \hat{y}^k + \rho_k^{-1} c(x)$, and the corresponding Newton system is
\begin{equation}
	\begin{bmatrix}
		H_k( x, y_k(x) ) + \frac{1}{\rho_k} \nabla c(x)^\top \nabla c(x)
	\end{bmatrix} \delta x = - \nabla_x \mathcal{L}_k(x,y_k(x)) ,
\end{equation}
where $H_k(x,y) \in \R^{n \times n}$ denotes the Hessian matrix $\nabla_{xx}^2 \mathcal{L}_k(x,y)$ or a symmetric approximation thereof.
A linear transformation yields the equivalent linear system
\begin{equation}
	\begin{bmatrix}
		H_k( x, y_k(x) ) & \nabla c(x)^\top \\
		\nabla c(x) & - \rho_k I
	\end{bmatrix} \begin{bmatrix}
		\delta x \\ \delta y
	\end{bmatrix} = - \begin{bmatrix}
		\nabla_x \mathcal{L}_k(x,y_k(x)) \\
		0
	\end{bmatrix} .
\end{equation}
Analogous Newton systems for \eqref{eq:PproxyBarrier} read
\begin{equation}
\begin{bmatrix}
		H_k(x,y) & \cdot & \nabla c(x)^\top \\
		\cdot &\rho_k I & - \rho_k I \\
		\nabla c(x) & - \rho_k I & \cdot
	\end{bmatrix} \begin{bmatrix}
		\delta x \\ \delta \lambda \\ \delta y
	\end{bmatrix} = - \begin{bmatrix}
		\nabla_x \mathcal{L}_k(x,y) \\
		\rho_k (\lambda - y) \\
		c(x) + \rho_k (\hat{y}^k - \lambda)
	\end{bmatrix}
\end{equation}
and formally solving for $\delta\lambda$ gives the condensed system
\begin{equation}
	\label{eq:linsys}
	\begin{bmatrix}
		H_k(x,y) & \nabla c(x)^\top \\
		\nabla c(x) & - \rho_k I
	\end{bmatrix} \begin{bmatrix}
		\delta x \\ \delta y
	\end{bmatrix} = - \begin{bmatrix}
		\nabla_x \mathcal{L}_k(x,y) \\
		c(x) + \rho_k (\hat{y}^k - y)
	\end{bmatrix} .
\end{equation}
The resemblances between these linear systems are apparent, as well as the differences.
The AL relaxation in \eqref{eq:Pbarrier_BCL} introduces a dual regularization for both the linear algebra and nonlinear solver, whose \emph{hidden} constraint $c(x) + \rho_k (\hat{y}^k - y) = 0$ holds pointwise due to the identity $y = y_k(x)$.
We remark that, entering the (2,2)-block, the dual regularization prevents issues due to linear dependence.
Furthermore, the primal regularization is left to the inertia correction strategy of the subsolver, affecting the (1,1)-block as in \cite[\S 3.1]{waechter2006implementation}.
If the approximation $H_k(x,y)$ is positive definite, e.g., by adopting suitable quasi-Newton techniques, the matrix in \eqref{eq:linsys} is symmetric quasi-definite and can be efficiently factorized with tailored linear algebra routines \cite{vanderbei1995symmetric}.

\section{Convergence Analysis}\label{sec:convergence}
In this section we analyze the asymptotic properties of the iterates generated by \cref{alg:RIPM} under the following blanket assumptions:
\begin{enumerate}[label=(\textsc{a}\arabic*)]
	\item Functions $\func{f}{\XX}{\R}$ and $\func{c}{\XX}{\YY}$ in \eqref{eq:P} are continuously differentiable.
	\item\label{ass:wellposed} Subproblems \eqref{eq:PproxyBarrier} are well-posed for all parameters' values, namely for any $\mu_k \leq \mu_0$, $\rho_k \leq \rho_0$, and $\hat{y}^k \in Y$.
\end{enumerate}
First, we characterize the iterates in terms of stationarity.
\begin{lem}\label{lem:xyzSignKKT}
	Consider a sequence $\{x^k,y^k,z^k\}$ generated by \cref{alg:RIPM}.
	Then, for all $k\in\N$, it is $x^k > 0$, $z^k \leq 0$, and the following conditions hold:
	\begin{subequations}\label{eq:xyzSignKKT}
		\begin{align}
			\| \nabla f(x^k) + \nabla c(x^k)^\top y^k + z^k \|
			{}\leq{}&
			\epsilon_k , \label{eq:xyzSignKKT:x}\\
			\| c(x^k) + \rho_k (\hat{y}^k - y^k) \|
			{}\leq{}&
			2 \epsilon_k \label{eq:xyzSignKKT:y}.
		\end{align}
	\end{subequations}
\end{lem}
\begin{proof}
	Positivity of $x^k$ follows from the barrier function $b$ having domain $\dom b = (0,\infty)^n$, whereas nonpositivity of $z^k$ is a consequence of $b_i^\prime \leq 0$ for all $i$ and $\mu_k > 0$.
	Based on \cref{defn:epsKKTstationary} and \cref{step:z} of \cref{alg:RIPM}, the $\epsilon_k$-KKT stationarity of $(x^k,\lambda^k)$ for \eqref{eq:PproxyBarrier}, with multiplier $y^k$, yields \eqref{eq:xyzSignKKT:x} along with
	\begin{subequations}\label{eq:epsKKTstationary:PproxyBarrier}
		\begin{align}
			\rho_k \norm{ \lambda^k - y^k } {}\leq{}& \epsilon_k , \label{eq:epsKKTstationary:PproxyBarrier:lambda}\\
			\norm{ c(x^k) + \rho_k (\hat{y}^k - \lambda^k) } {}\leq{}& \epsilon_k . \label{eq:epsKKTstationary:PproxyBarrier:y}
		\end{align}
	\end{subequations}
	By the triangle inequality, \eqref{eq:epsKKTstationary:PproxyBarrier:lambda}--\eqref{eq:epsKKTstationary:PproxyBarrier:y} imply \eqref{eq:xyzSignKKT:y}.
	\qedhere
\end{proof}
Patterning \cite[Thm 4.2(ii)]{demarchi2022interior}, we establish asymptotic complementarity.
\begin{lem}\label{lem:asympComplementarity}
	Consider a sequence $\{ x^k,y^k,z^k \}$ of iterates generated by \cref{alg:RIPM} with $\varepsilon = 0$.
	Then, it holds $\lim\limits_{k\to\infty} \min\{x^k,-z^k\} = 0$.
\end{lem}
\begin{proof}
	The algorithm can terminate in finite time only if the returned triplet $(x^\ast,y^\ast,z^\ast)$ satisfies $\min\{ x^\ast,-z^\ast \} = 0$.
	Excluding this ideal situation, we may assume that it runs indefinitely and that consequently $\mu_k \searrow 0$, owing to \cref{step:updateBarrier:if,step:updateBarrier:failed} and recalling that $x^k > 0$ and $z^k \leq 0$ for all $k\in\N$ by \cref{lem:xyzSignKKT}.
	Consider now an arbitrary index $i\in\{1,\ldots,n\}$ and the two possible cases.
	If $x_i^k \to 0$, then the statement readily follows from $z_i^k \leq 0$.
	If instead a subsequence $\{x_i^k\}_K$ remains bounded away from zero, then $\{ b_i^\prime(x_i^k) \}_K$ is bounded and therefore $z_i^k = \mu_k b_i^\prime(x_i^k) \to 0$ as $k\to_K \infty$, proving the statement since $x_i^k>0$.
	The claim then follows from the arbitrarity of the index $i$ and the subsequence.
	\qedhere
\end{proof}
Like all penalty-type methods in the nonconvex setting, \cref{alg:RIPM} may generate limit points that are infeasible for \eqref{eq:P}.
Patterning standard arguments, the following result gives sufficient conditions for the feasibility of limit points; cf. \cite[Ex. 4.12]{birgin2014practical}.
\begin{prop}
	Consider a sequence $\{ x^k,y^k,z^k \}$ of iterates generated by \cref{alg:RIPM}.
	Then, each limit point $x^\ast$ of $\{x^k\}$ is feasible	for \eqref{eq:P} if one of the following conditions holds:
	\begin{enumerate}[label=(\roman*)]
		\item the sequence $\{\rho_k\}$ is bounded away from zero, or
		\item there exists some $B \in \R$ such that for all $k\in\N$
			\begin{equation*}
				f(x^k) + \frac{1}{2\rho_k} \norm{c(x^k) + \rho_k \hat{y}^k}^2 \leq B .
			\end{equation*}
	\end{enumerate}
\end{prop}
These conditions are generally difficult to check a priori.
Nevertheless, in the situation where each iterate $x^k$ is actually a (possibly inexact) global minimizer of \eqref{eq:PproxyBarrier}, then limit points generated by \cref{alg:RIPM} have minimum constraint violation and tend to minimize the objective function subject to minimal infeasibility \cite[Thm 5.1, Thm 5.3]{birgin2014practical}.
In particular, limit points are indeed feasible if \eqref{eq:P} admits feasible points.
However, these properties cannot be expected by solving the subproblems only up to stationarity.
Nonetheless, even in the case where a limit point is not necessarily feasible, the next result shows that it is at least a stationary point for a feasibility problem associated to \eqref{eq:P}.
\begin{prop}
	Consider a sequence $\{x^k,y^k,z^k\}$ generated by \cref{alg:RIPM} with $\varepsilon = 0$.
	Then each limit point $x^\ast$ of $\{x^k\}$ is KKT stationary for the problem
	\begin{equation}\label{eq:feasibilityP}
		\minimize_{x\geq 0} \quad \frac{1}{2} \|c(x)\|^2 .
	\end{equation}
\end{prop}
\begin{proof}
	We may consider two cases, depending on the sequence $\{\rho_k\}$.
	If $\{\rho_k\}$ remains bounded away from zero, then \cref{step:updatePenalty:if,step:updatePenalty:failed} of \cref{alg:RIPM} imply that $\|c(x^k)\|\to0$ for $k\to\infty$.
	Continuity of $c$ and properties of norms yield $c(x^\ast) = 0$.
	Furthermore, by construction, we have $x^k > 0$ for all $k\in\N$, hence $x^\ast \geq 0$.
	Altogether, this shows that $x^\ast$ is feasible for \eqref{eq:P}, namely a global minimizer for the feasibility problem \eqref{eq:feasibilityP} and, therefore, a KKT stationary point thereof.
	Assume now that $\rho_k \to 0$.
	Define $\delta^k\in\XX$ and $\eta^k \in \YY$ as
	\begin{align*}
		\delta^k {}\coloneqq{}& \nabla f(x^k) + \nabla c(x^k)^\top y^k + z^k \\
		\eta^k {}\coloneqq{}& c(x^k) + \rho_k (\hat{y}^k - y^k)
	\end{align*}
	for all $k\in\N$.
	In view of \cref{lem:xyzSignKKT}, we have that $\norm{\delta^k} \leq \epsilon_k$ and $\norm{\eta^k}\leq2\epsilon_k$ hold for all $k\in\N$.
	Multiplying $\delta^k$ by $\rho_k$, substituting $y^k$ and rearranging, we obtain
	\begin{equation*}
		\rho_k \delta^k
		{}={}
		\rho_k \nabla f(x^k) + \nabla c(x^k)^\top \left[ \rho_k \hat{y}^k + c(x^k) - \eta^k \right] + \rho_k z^k .
	\end{equation*}
	Now, let $x^\ast$ be a limit point of $\{x^k\}$ and $\{x^k\}_K$ a subsequence such that $x^k \to_K x^\ast$.
	Then the sequence $\{\nabla f(x^k)\}_K$ is bounded, and so is $\{\hat{y}^k\}_K \subset Y$ by construction.
	Recalling from \cref{lem:xyzSignKKT} that $x^k > 0$ and $z^k \leq 0$, and observing that $0 \leq \norm{\delta^k}, \norm{\eta^k} \leq 2 \epsilon_k \to 0$, we shall now take the limit of $\rho_k \delta^k$ for $k\to_K\infty$, resulting in
	\begin{equation*}
		0
		{}={}
		\nabla c(x^\ast)^\top c(x^\ast) + \tilde{z}^\ast
	\end{equation*}
	for some $\tilde{z}^\ast \leq 0$.
	As a limit point of $\{ \rho_k z^k\}$, $\tilde{z}^\ast$ together with $x^\ast$ satisfy $\min\{x^\ast,-\tilde{z}^\ast\}=0$ by \cref{lem:asympComplementarity}.
	Since we also have $x^\ast \geq 0$, it follows that $x^\ast$ is KKT stationary for \eqref{eq:feasibilityP} according to \cref{defn:epsKKTstationary}.
	\qedhere
\end{proof}
Finally, we qualify the output of \cref{alg:RIPM} in the case of feasible limit points.
In particular, it is shown that any feasible limit point is AKKT stationary for \eqref{eq:P} in the sense of \cref{defn:asympKKTstationary}.
Under some additional boundedness conditions, feasible limit points are KKT stationary, according to \cref{defn:epsKKTstationary}.
\begin{thm}\label{thm:subseqKKTstationarity}
	Let $\{x^k,y^k,z^k\}$ be a sequence of iterates generated by \cref{alg:RIPM} with $\varepsilon = 0$.
	Let $x^\ast$ be a feasible limit point of $\{x^k\}$ and $\{x^k\}_K$ a subsequence such that $x^k \to_K x^\ast$.
	Then,
	\begin{enumerate}[label=(\roman{*})]
		\item\label{thm:subseqKKTstationarity:asymptotic} $x^\ast$ is an AKKT stationary point for \eqref{eq:P}.
		\item\label{thm:subseqKKTstationarity:bounded} If $\{y^k,z^k\}_K$ remain bounded, then $x^\ast$ is KKT stationary for \eqref{eq:P}.
	\end{enumerate}
\end{thm}
\begin{proof}
	\ref{thm:subseqKKTstationarity:asymptotic}
		Together with the fact that $\epsilon_k \to 0$, \cref{lem:xyzSignKKT} ensures that the sequence $\{x^k\}_K$ satisfies condition \eqref{eq:asympKKTstationary:x}, whereas \cref{lem:asympComplementarity} implies \eqref{eq:asympKKTstationary:z}.
		Feasibility of $x^\ast$ completes the proof.
	
	\ref{thm:subseqKKTstationarity:bounded}
		By boundedness, the subsequences $\{y^k\}_K$ and $\{z^k\}_K$ admit some limit points $y^\ast$ and $z^\ast$, respectively.
		Thus, from the previous assertion and with continuity arguments on $f$ and $c$, it follows that $x^\ast$ is KKT stationary for \eqref{eq:P}, not only asymptotically.
	\qedhere
\end{proof}
Provided that the iterates admit a feasible limit point, finite termination of \cref{alg:RIPM} with an $\varepsilon$-KKT stationary point can be established as a direct consequence of \cref{thm:subseqKKTstationarity}.

\section{Numerical Results}\label{sec:numerics}
In this section we test an instance of the proposed regularized interior point approach, denoted \regip{}, on the CUTEst benchmark problems \cite{gould2015cutest}.
\regip{} is compared in terms of robustness against the IP solver \ipopt{} \cite{waechter2006implementation} and the AL solver \percival{} \cite{santos2020percivaljl}, which is based on a BCL method \cite{conn1991globally} coupled with a trust-region matrix-free solver \cite{lin1999newton} for the subproblems.
We do not report runtimes nor iteration counts since a fair comparison would require close inspection of heuristics and fallbacks \cite[\S 3]{waechter2006implementation}.

We implemented \regip{} in Julia and set up the numerical experiments adopting the JSO software infrastructure by \cite{orban2019jso}.
The IP solver \ipopt{} acts as subsolver to execute \cref{step:xy}, warm-started at the current primal $(x^{k-1}, y^{k-1})$ and dual $(y^{k-1}$, $z^{k-1})$ estimates.
We use its parameter \ipoptOption{tol} to set the (inner) tolerance $\epsilon_k$, disabling other termination conditions, and let \ipopt{} control the barrier parameter as needed to approximately solve the regularized subproblem.%
\footnote{Solving a sequence of barrier subproblems may hinder the computational efficiency of \regip{} compared to the approach behind \cref{alg:RIPM}, but does not degrade its reliability.
Ongoing research focuses on solving \eqref{eq:PproxyBarrier} and letting the IP subsolver update the barrier parameter after warm-starting at the current primal-dual estimate, in the spirit of \cite[Alg. 3]{cipolla2022proximal}.\label{foot:barriersubsolver}}
We let the safeguarding set be $Y \coloneqq \{ v\in\YY \,|\, \|v\|_\infty \leq 10^{20} \}$ and choose $\hat{y}^k$ by projecting the current estimate $y^{k-1}$ onto $Y$.
We set the initial penalty parameter to $\rho_0 = 10^{-6}$, the inner tolerance $\epsilon_0 = \sqrt[3]{\varepsilon}$, and parameters $\theta_\rho = 0.5$, $\kappa_\rho = 0.5$, and $\kappa_\epsilon = 0.5$.
\regip{} declares success, and returns a $\varepsilon$-KKT stationary point, as soon as $\epsilon_k \leq \varepsilon$ and $C^k \leq \varepsilon$.%
\footnote{The condition $\epsilon_k \leq \varepsilon$ implies both $\varepsilon$-stationarity \eqref{eq:epsKKTstationary:x} and $\varepsilon$-complementarity \eqref{eq:epsKKTstationary:z} required for $\varepsilon$-KKT stationarity. This follows from the observation that, in \regip{}, the subsolver approximately solves \eqref{eq:Pproxy} at \cref{step:xy}, not \eqref{eq:PproxyBarrier}; see Footnote~\ref{foot:barriersubsolver}.}
Instead, if $\epsilon_k \leq \varepsilon$, $C^k > \varepsilon$ and $\rho_k \leq \rho_{\min} \coloneqq 10^{-20}$, \regip{} stops declaring (local) infeasibility.
For \ipopt{}, we set the tolerance \ipoptOption{tol} to $\varepsilon$, remove the other desired thresholds, and disable termination based on acceptable iterates.
For \percival{}, we set absolute and relative tolerances \ipoptOption{atol}, \ipoptOption{rtol}, and \ipoptOption{ctol} to $\varepsilon$.
We select all CUTEst problems with at most 1000 variables and constraints, obtaining a test set with 609 problems.
All solvers are provided with the primal-dual initial point available in CUTEst, a time limit of $60$ seconds, the maximum number of iterations set to $10^9$, and a tolerance $\varepsilon \in \{10^{-3}, 10^{-5} \}$.
A solver is deemed to solve a problem if it returns with a successful status; it fails otherwise.
The source codes for the numerical experiments have been archived on Zenodo at
\begin{center}
	\href{https://doi.org/\ZenodoCodeDoi}{\textsc{doi}: \ZenodoCodeDoi}.
\end{center}

\Cref{tab:cutest} summarizes the results, stratified by solver, termination tolerance ($\varepsilon$) and problem size ($n$, $m$).
For each combination, we indicate the number of times \regip{} solves a problem that the other solver fails (``W'') or solves (``T+'') and the number of times \regip{} fails on a problem that the other one fails (``T-'') or solves (``L'').
The results show that \regip{} succeeds on more problems than the other solvers, consistently for both low and high accuracy, indicating that the underlying regularized IP approach can form the basis for reliable and scalable solvers.%
\footnote{We agree with \cite[\S 5]{birgin2016sequential} on the fact that ``strong statements concerning the relative efficiency or robustness [$\ldots$] are not possible in nonlinear optimization.''}

\begin{table}
	\begin{center}
		\renewcommand{\arraystretch}{1.2}
		\caption{Comparison on CUTEst problems with $n$ variables and $m$ constraints}%
		\label{tab:cutest}%
		\begin{tabular}{cc|ccc|ccc}
			\hline
			\multicolumn{8}{c}{\regip{} against \ipopt{}} \\ \hline
			\multicolumn{2}{c|}{Size $n, m$} & \multicolumn{3}{c|}{Tolerance $\varepsilon = 10^{-3}$} & \multicolumn{3}{|c}{Tolerance $\varepsilon = 10^{-5}$} \\
			Min & Max & W & T & L & W & T & L \\ \hline
			0 & 10 & 19 & 417+/25- & 3 & 15 & 417+/29- & 3 \\
			11 & 100 & 6 & 60+/7- & 1 & 7 & 58+/8- & 1 \\
			101 & 1000 & 6 & 57+/8- & 0 & 7 & 53+/10- & 1 \\ \hline
			\multicolumn{8}{c}{\regip{} against \percival{}} \\ \hline
			\multicolumn{2}{c|}{Size $n, m$} & \multicolumn{3}{c|}{Tolerance $\varepsilon = 10^{-3}$} & \multicolumn{3}{|c}{Tolerance $\varepsilon = 10^{-5}$} \\
			Min & Max & W & T & L & W & T & L \\ \hline
			0 & 10 & 17 & 419+/20- & 8 & 19 & 413+/24- & 8 \\
			11 & 100 & 12 & 54+/8- & 0 & 18 & 47+/8- & 1 \\
			101 & 1000 & 37 & 26+/8- & 0 & 37 & 23+/11- & 0 \\ \hline
		\end{tabular}
	\end{center}
\end{table}

\section{Conclusion}\label{sec:conclusion}
This paper has presented a regularized interior point approach to solving constrained nonlinear optimization problems.
Operating as an outer regularization layer, a quadratic proximal penalty provides robustness whilst consuming minimal computation effort once embedded into existent interior point codes as a principled inertia correction strategy.
Furthermore, regularizing the equality constraints allows to safely adopt more efficient linear algebra routines, while waiving the need for an infeasibility detection mechanism within the subsolver.
Preliminary numerical results indicate that a close integration of proximal regularization within interior point schemes is key to provide efficient and robust solvers.
We encourage further research in this direction.

\subsection*{Acknowledgements}
\TheAcknowledgements

\phantomsection
\addcontentsline{toc}{section}{References}
\bibliographystyle{jnsao}
{\small\bibliography{biblio}}

\end{document}